\documentclass[12pt]{amsart}
\usepackage{txfonts}
\usepackage{amscd,amssymb, amsmath, eucal}
\usepackage[colorlinks,plainpages,backref,urlcolor=blue]{hyperref}
\usepackage{verbatim}
\usepackage{xcolor}

\newtheorem{theorem}{Theorem}[section]
\newtheorem{corollary}[theorem]{Corollary}
\newtheorem{lemma}[theorem]{Lemma}
\newtheorem{prop}[theorem]{Proposition}

\newtheorem{innercustomthm}{Theorem}
\newenvironment{customthm}[1]
  {\renewcommand\theinnercustomthm{#1}\innercustomthm}
  {\endinnercustomthm}

\theoremstyle{definition}

\newtheorem{example}[theorem]{Example}
\newtheorem{remark}[theorem]{Remark}

\newcommand{\Z}{\mathbb{Z}}

\newcommand{\C}{\mathbb{C}}

\newcommand{\PP}{\mathbb{P}}

\renewcommand{\k}{\Bbbk}

\DeclareMathAlphabet{\pazocal}{OMS}{zplm}{m}{n}
\newcommand{\A}{{\pazocal{A}}}

\newcommand{\OO}{{\pazocal O}}

\renewcommand{\k}{\Bbbk}

\newcommand{\cE}{{\mathcal{E}}}
\newcommand{\cT}{{\mathcal{T}}}

\newcommand{\CC}{{\mathcal{C}}}

\DeclareMathOperator{\Irr}{Irr}

\def\dot{\mathchar"013A}
\newcommand{\hdot}{{\raise1pt\hbox to0.35em{\Huge $\dot$}}}

\begin{document}

\title[Addition-deletion for conic-line arrangements with split Chern polynomial]%
{Addition-deletion for conic-line arrangements with split Chern polynomial}

\author[A. M\u acinic]{Anca~M\u acinic$^*$}
\thanks{$^*$ Partially supported by a grant of the Ministry of Research, Innovation and Digitization, CNCS - UEFISCDI, project number PN-IV-P1-PCE-2023-2001, within PNCDI IV}
\address{Simion Stoilow Institute of Mathematics, 
 Bucharest, Romania}
\email{Anca.Macinic@imar.ro}

\author[J. Vall\`es]{Jean Vall\`es$^{**}$}
\thanks{$^{**}$ Partially supported by ANR
BRIDGES ANR-21-CE40-0017.}
\address{Université de Polynésie française, Polynésie française}
\email{jean.valles@upf.pf}
\date{\today}


\begin{abstract} 
We present combinatorial/geometric obstructions induced by the  factorization over the integers of the Chern polynomial of the bundle of logarithmic vector fields associated to a complex projective plane curve. Our results generalize at the same time similar results on projective lines arrangements whose characteristic polynomial factors over the integers and results on free curves. We give a splitting criterion for a rank 2 vector bundle, in terms of restrictions to smooth conics.
\end{abstract}
 
\maketitle

\section{Introduction}

Our objects of study in this note will be projective plane curves, in particular arrangements of 
 lines in the complex projective plane $\PP^2=\PP^2\C$. All curves will be 
 assumed to be reduced.
 We will give an overview of  the general context and the main results, while precise definitions of the notions involved are found in section  \S~\ref{sec:preliminaries}.

Abe shows in  \cite{A2} that the characteristic polynomial of an arrangement of projective lines
induces restrictions on the number of intersection points of the lines in the arrangement, and on the number of intersection points of the arrangement to an arbitrary line in $\PP^2$.  In \cite{FV},  Faenzi and the second author give results of the same flavour, but in terms of the Chern polynomial $c_{\mathcal{T}_{\A}}(t) $ of the bundle of logarithmic vector fields $\cT_{\A}$ associated to the arrangement $\A$ (see section \ref{sec:preliminaries} for a definition).
It is shown in \cite{MS} 
that, for locally free arrangements, in particular  for projective  lines arrangements,  the characteristic  polynomial of the arrangement $\A$  and the Chern polynomial $c_{\mathcal{T}_{\A}}(t) $ determine one another (the relation is to be understood via the Poincar\'e polynomial of the complement of the arrangement, see \cite{OT}).
For this reason, we will recall here condensed versions of the results from \cite{A2} and \cite{FV} phrased  in terms of the Chern polynomial, in Theorems \ref{thm:abe_FV_1},  \ref{thm:abe_FV_2}. 

\begin{theorem}
\label{thm:abe_FV_1}
Let $\A$ be an arrangement of projective lines such that, for some $a,b \in \Z_{>0}, a \leq b$, we have $$c_{\mathcal{T}_{\A}}(t) = (at-1)(bt-1).$$  {\it (i)} For $L \in \A$, denote $\A':= \A \setminus \{L\}$. Then:
\begin{enumerate}
\item  For any $L \in \A$, $|\A' \cap L| \leq a+1$ or $|\A' \cap L| \geq b+1$.
\item If there is a line $L \in \A$ such that $|\A' \cap L| \in \{a+1, b+1\}$, then $\A$ is free.
\item If  there is a line $L \in \A$ such that $|\A' \cap L| > b+1$, then $\A$ is not free.
\end{enumerate}
{\it (ii)} For any line $L \subset \PP^2, \; L \notin \A$, either $|\A \cap L|  \leq a+1$ or $|\A \cap L|  \geq b+1$.
\end{theorem}

When the arrangement $\A$ is free with exponents $(a,b)$ we have a factorization of the Chern polynomial  $c_{\mathcal{T}_{\A}}(t) = (at-1)(bt-1)$,
hence the above result applies in particular to free arrangements. Moreover, for free arrangements, a stronger statement holds:

\begin{theorem}
\label{thm:abe_FV_2}
 Let $\A$ be a free projective lines arrangement with exponents $(a,b), \; a \leq b$.
  {\it (i)} Let $L \in \A$ and  $\A' := \A \setminus \{L\}$. Then $|\A' \cap L| \leq a+1$ or $|\A' \cap L| = b+1$.
  
 \noindent  {\it (ii)} For any $L \notin \A$, either $|\A \cap L| = a+1$ or $|\A \cap L| \geq b+1$.
\end{theorem}

The above result is  generalized to arbitrary reduced curves in the complex projective plane in \cite[Theorem 2.6]{M}. This generalization, recalled in Theorem \ref{thm:curve_line_comb} below, takes into account the complexity of the singularities of a curve, via the invariant $\epsilon(\cdot, \cdot)$ introduced in \cite{Dimca}.
If $\CC$ is a curve, we will denote by  $\Irr(\CC)$ the set of irreducible components of $\CC$. For the definition of the invariant $\epsilon(\cdot, \cdot)$ see section \S~\ref{sec:preliminaries}.

\begin{theorem} 
\label{thm:curve_line_comb}
Let $\CC$ be a reduced free curve in $\PP^2$ with exponents $(a,b), a \leq b$.

 {\it (i)} If  $L \in \Irr(\CC)$ is a line and $\CC' \coloneqq \bigcup_{C \in \Irr(\CC) \setminus \{L\}}C$, then either 
$|\CC' \cap L| \leq  a+1 - \epsilon(\CC,\CC')$ or  $|\CC' \cap L| = b+1 - \epsilon(\CC,\CC')$.
 
 {\it (ii)} If $L \subset \PP^2$ is a line such that $L \notin \Irr(\CC)$, then either $|\CC \cap L|= a+1  - \epsilon(\CC \cup L, \CC)$ or $|\CC \cap L|  \geq b+1 - \epsilon(\CC \cup L,\CC)$.
\end{theorem} 

In our first main result, we extend the statements of  Theorem \ref{thm:curve_line_comb} under a weaker hypothesis than freeness, a hypothesis on the factorization of the Chern polynomial.This result generalizes to curves Theorem \ref{thm:abe_FV_1}.

 \begin{customthm}{A}
\label{thm:curve_line_comb_chernA}
Let $\CC$ be a reduced curve in $\PP^2$ such that 
$c_{\mathcal{T}_{\CC}}(t) = (at-1)(bt-1),\; a,b \in \Z_{>0},$
$ a \leq b$. 
\begin{enumerate}
    \item Let  $L \in  \Irr(\CC)$ be a line and $\CC' \coloneqq \bigcup_{C \in \Irr(\CC) \setminus\{ L\} }C$.
Then:
\begin{enumerate}
\item  $|\CC' \cap L| \leq a+1  -  \epsilon(\CC,\CC')$ or  $|\CC' \cap L| \geq b+1  -  \epsilon(\CC,\CC')$;
\item If  $|\CC' \cap L| = a+1 - \epsilon(\CC,\CC') $ or $|\CC' \cap L|  = b+1- \epsilon(\CC,\CC')$, then $\CC$ is free with exponents $(a,b)$;
\item  If $|\CC' \cap L|  > b+1 - \epsilon(\CC,\CC')$, then $\CC$ is not free.
\end{enumerate}
\item  Let $L \subset \PP^2$ be a line such that $L \notin \Irr(\CC)$. Then: 
\begin{enumerate}
\item $|\CC \cap L|\leq  a+1  - \epsilon(\CC \cup L, \CC)$ or $|\CC \cap L|  \geq b+1 - \epsilon(\CC \cup L, \CC)$;
\item If  $|\CC \cap L| = a+1 - \epsilon(\CC \cup L, \CC) $ or $|\CC \cap L|  = b+1- \epsilon(\CC \cup L, \CC)$, then $\CC$ is free with exponents $(a,b)$;
\item If $|\CC \cap L|  < a+1 - \epsilon(\CC \cup L, \CC)$, then $\CC$ is not free.
\end{enumerate}
\end{enumerate}
\end{customthm}

Moreover, we construct an example which shows that Theorem \ref{thm:curve_line_comb} does not hold if we replace the freeness condition in the hypothesis by the factorization of the associated Chern polynomial,  see  Example \ref{ex:chern_factor}.\\

There is a growing interest in the study of a natural generalization of projective lines arrangements inside the class of projective plane curves, the so called arrangements of conics and lines (or CL arrangements), see for instance \cite{ST, BMR, Pk}. They are defined as curves whose irreducible components are smooth irreducible curves of genus $0$, i.e. lines and smooth conics. This brings up the question whether a result similar to Theorem \ref{thm:curve_line_comb_chernA}  holds when we consider a smooth conic $C_0$ instead of a line $L$. More precisely,  if one can give obstructions on the number of singular points of a curve $\CC$, situated on a smooth conic  $C_0 \in  \Irr(\CC)$ and on the number of intersection points of $\CC$ to an arbitrary smooth conic  $C_0$,  under certain conditions on the curve $\CC$.
In \cite{M} one gives a positive answer to this question, when $\CC$ is a free curve: 
\begin{theorem}
\label{thm:curve_conic_comb_intro}
Let $\CC$ be a reduced free curve  in $\PP^2$ with exponents $(a,b), \; a \leq b$.

 {\it (i)} If $C_0 \in \Irr(\CC)$ is an  arbitrary smooth conic, let
 $\CC' \coloneqq \bigcup_{C \in \Irr(\CC) \setminus\{C_0\} }C$
 and $k =|\CC' \cap C_0|+\epsilon(\CC, \CC')$. 
\begin{enumerate}
\item  If $k=2m$, then the only possible values for $m$ are $m=b$ or $m \leq a$.
\item $k=2m+1$, then $m=a-1=b-1$ or $m \leq a-1$.
\end{enumerate}
 {\it (ii)} If  $C_0 \notin  \Irr(\CC)$ is an arbitrary  smooth conic,  let $k = |\CC \cap C_0|+\epsilon(\CC \cup C_0, \CC)$. 
\begin{enumerate}
\item If $k=2m$, then either $m=a$ or $m \geq b$.
\item If $k=2m+1$, then either $m=a=b$ or $m \geq b$.
\end{enumerate}
\end{theorem}

 In a similar vein to Theorem \ref{thm:curve_line_comb_chernA}, we may wonder whether some of the conclusions of  Theorem \ref{thm:curve_conic_comb_intro} hold if we replace the freeness hypothesis by the factorization of the Chern polynomial.
 We answer this question positively in our next two main results.
\begin{customthm}{B1}
\label{thm:curve_conic_comb_chern} 
Let $\CC$ be a reduced curve in $\PP^2$ such that 
$c_{\mathcal{T}_{\CC}}(t) = (at-1)(bt-1),\; a,b \in \Z_{>0},$
$ a \leq b$. Let $C_0  \in  \Irr(\CC)$ be a smooth conic, $\CC' \coloneqq \bigcup_{C \in \Irr(\CC) \setminus\{C_0\} }C$ and let
$k :=|\CC' \cap C_0|+\epsilon(\CC, \CC')$. Then:
\begin{enumerate}
\item  If $k=2m$, either $m\geq b$ or $m \leq a$. Moreover, 
\begin{enumerate}
    \item If $m=a$ or $m=b$ then $\CC$ is free with exponents $(a,b)$;
    \item If $m<a$ both possibilities, free and non free, occur;
    \item If  $m>b$ then $\CC$ is not free.
\end{enumerate}
\item If $k=2m+1$, either $m \leq a-1$ or $m \geq b-1$. Moreover,
\begin{enumerate}
    \item If $m = a -1$ or $m = b - 1$, then
$\CC$ is free with exponents $(a, b)$; if, furthermore, $m = b - 1$, then $a = b = m + 1$;
\item If $m<a-1$  then both possibilities, free and non free, occur;
\item  If $m>b-1$ then $\CC$ is not free.
\end{enumerate}
\end{enumerate}
\end{customthm}

\begin{customthm}{B2}
\label{thm:curve_conic_comb_chern_add}
Let $\CC$ be a reduced curve in $\PP^2$ such that 
$c_{\mathcal{T}_{\CC}}(t) = (at-1)(bt-1),\; a,b \in \Z_{>0},$
$ a \leq b$. Let $C_0  \notin \Irr(\CC)$ be a smooth conic and let  $k :=|\CC \cap C_0|+\epsilon(\CC \cup C_0, \CC)$. Then:
\begin{enumerate}
\item  If $k=2m$, either $m\leq a$ or $m \geq b$. 

Moreover, 
\begin{enumerate}
    \item If $m=a$ or $m=b$ then $\CC$ is free with exponents $(a,b)$;
    \item If $m<a$ then $\CC$ is not free;
    \item If  $m>b$ both possibilities, free and non free, occur.
    \end{enumerate}

\item $k=2m+1$, either $m \leq a$ or $m \geq b$. Moreover,

\begin{enumerate}
    \item If $m = a$ then necessarily  $m=b$ and $\CC$ is free with exponents $(a,a)$;
    \item If $m=b$ then $\CC$ is free with exponents $(a,b)$;
\item If $m<a$   then $\CC$ is not free; 
\item  If $m>b$ then both possibilities, free and non free, occur.
\end{enumerate}
\end{enumerate}

\end{customthm}

The proofs of Theorems \ref{thm:curve_line_comb_chernA}, \ref{thm:curve_conic_comb_chern} and \ref{thm:curve_conic_comb_chern_add} rely on the existence of short exact sequences of vector bundles arising from addition-deletion operations of smooth rational curves (i.e. lines and smooth conics), applied to an arbitrary curve.
\\

We propose a new approach to the addition-deletion of a smooth conic to a curve, in subsection \S~ \ref{subsect:add_dell}. The advantage of this approach is that the exact sequence we use has a parameter that depends solely on Tjurina numbers, as opposed to \cite{M}, where the parameter involved in the conditions that describe the result of addition-deletion is defined in terms of both Tjurina and Milnor numbers, see \cite[Theorems 3.14 and 3.16]{M}.  
Moreover, this yields a formula that involves the Milnor numbers of the singularities of a curve $\CC$ and the Milnor numbers of the singularities of its addition $\CC \cup C_0$ that are situated on the smooth conic $C_0 \notin \Irr(\CC)$. This formula is stated in Proposition \ref{prop:mu_formula}. To the best of our knowledge, this formula was not previously known.

\smallskip

Our last main results, proved in section  \S~ \ref{sect:conic}, are two general results on rank $2$ vector bundles that are key in the proofs of Theorems \ref{thm:curve_conic_comb_chern} and \ref{thm:curve_conic_comb_chern_add}. The first one describes the splitting type of the restriction to a smooth conic of an arbitrary vector bundle, under some hypothesis on its Chern classes.
\begin{customthm}{C}
\label{thm:conic_restr}
Let $\mathcal{E}$ be a rank $2$ vector bundle on $\PP^2$  such that $c_1(\mathcal{E}) = c \leq 0$ and $c_2(\mathcal{E})  =0$. Then, for any smooth conic $C \subset \PP^2$, there exists a positive integer $r$, more precisely $r=0$ and in this case $\mathcal{E}=\OO_{\PP^2} \oplus \OO_{\PP^2}(c)$ or $r\ge 2$, depending on $C$, such that the restriction of $\mathcal{E}$ to $C$ splits as
 $$\mathcal{E}|_ {C}= \mathcal{O}_{C}(\frac{r}{2}) \oplus \mathcal{O}_{C}(c-\frac{r}{2}).$$
 \end{customthm}

Finally we give a freeness criterion for rank $2$ vector bundles, in terms of restrictions of the vector bundle to smooth conics, in line with Yoshinaga's freeness criterion relative to the restriction of a rank $2$ vector bundle to a line (recalled in Theorem \ref{thm:pi_L}).
 
\begin{customthm}{D}
\label{thm:conic_criterion}
Let $\cE$ be a rank $2$ vector bundle on $\PP^2$ such that  $c_{\mathcal{E}}(t) = (at-1)(bt-1)$.
 Then there exists a smooth conic $C_0 \subset \PP^2$ such that  $\cE|_{C_0} =\OO_{C_0}(-a) \oplus  \OO_{C_0}(-b)$ if and only if  $\cE$ splits as $\cE =\OO_{\PP^2}(-a) \oplus  \OO_{\PP^2}(-b).$
\end{customthm}

\section{Preliminaries}
\label{sec:preliminaries}

Let $\CC$ be a reduced curve in $\PP^2 \coloneqq \PP^2 \C$ defined as the zero set of the homogeneous polynomial $f_{\CC} \in S:= \C[x,y,z]$. Denote by $D_0(\CC)$  the logarithmic derivations module defined by
$$ 
D_0(\CC) \coloneqq \{\theta \in Der(S) \; |\; \theta(f_{\CC}) = 0\}.
$$
and by $\mathcal{T}_{\CC}$ the sheafification of the derivations module $D_0(\CC)$.\\

The curve $\CC$ is called free if $D_0(\CC)$ is free as an $S$-module. This is equivalent to saying that the associated vector bundle, $\mathcal{T}_{\CC}$, is free, i.e., it splits as a direct sum of line bundles.

Let us recall a well known freeness criterion for rank $2$ vector bundles on $\PP^2$.

\begin{theorem}(\cite[Theorem 1.45]{Y})
\label{thm:pi_L}
Let $\cE$ be a rank $2$ vector bundle on $\PP^2$ and $L \subset \PP^2$ a line. Let $\cE|_{L} =\OO_L(-e_1) \oplus  \OO_L(-e_2)$. Then $c_2(\cE) \geq e_1 e_2$, furthermore 
$$
c_2(\cE) - e_1 e_2 = \dim_{\k}\mathrm{Coker}(\Gamma_*(\cE) \overset{\pi_L}{\longrightarrow}(\Gamma_*(\cE|_{L}))
$$
where $c_2(\cE)$ is the second Chern number of $\cE$ and $\pi_L$ is the morphism of graded modules induced by the restriction to $L$ of the vector bundle $\cE$. Moreover, $\cE$ is splitting if and only if $c_2(\cE) = e_1 e_2$.
\end{theorem}

  If  $\CC$ is an arbitrary reduced curve in $\PP^2$ and $p \in \mathrm{Sing}(\CC)$, denote by $\mu_p (\CC),\; \tau_p(\CC)$, the Milnor number, respectively the Tjurina number, of the singularity $p$.  Let 
 $$
 \epsilon_p(\CC) \coloneqq \mu_p(\CC) - \tau_p (\CC)
 $$
One can see $\epsilon_p(\CC)$ as a measure of  the defect from quasihomogeneity of $p \in \mathrm{Sing}(\CC)$, since  $p$ is quasihomogeneous if and only if the Milnor and Tjurina numbers are equal. 
\\

If  $\CC_1, \; \CC_2$ are two reduced curves  with no common irreducible component and  $p \in \CC_1 \cap  \CC_2$, denote
$$
 \epsilon(\CC_1\cup \CC_2, \CC_1)_p
 \coloneqq  \epsilon_p(\CC_1 \cup  \CC_2) -  \epsilon_p(\CC_1)
$$
and 
$$
 \epsilon(\CC_1\cup \CC_2, \CC_1) \coloneqq \sum_{p \in \CC_1 \cap  \CC_2} 
 \epsilon(\CC_1\cup \CC_2, \CC_1)_p.
$$
 Please notice that our notation for the invariant $\epsilon(\cdot,\cdot)$, which was introduced in \cite{Dimca}, is slightly different from the one in \cite{Dimca} and \cite{M}: $\epsilon(\CC_1\cup \CC_2, \CC_1)$ was denoted as $\epsilon(\CC_1, \CC_2)$ in the previously mentioned articles.
 
Let $\CC_2: f_2 =0$.  In what follows, assume moreover that $\CC_2$ is smooth of genus $g_2$. Then, by \cite[Theorem 2.3]{Dimca} (see also \cite{STY}), there exists a short exact sequence of logarithmic sheaves (actually, vector bundles) on $\PP^2$:

\begin{equation}
\label{eq:sequence_innitial}
0 \rightarrow \mathcal{T}_{\CC_1}(1-\deg(f_2)) \overset{f_2}\longrightarrow  \mathcal{T}_{\CC_1 \cup \CC_2}(1) \rightarrow i_*(\mathcal{O}_{\CC_2}(D))\rightarrow 0.
\end{equation}

where $i: C_2 \rightarrow \PP^2$ is the inclusion. Moreover, the degree of the divisor $D$ is given by the following formula, where $r$ is the number of points in the reduced scheme of $\CC_1 \cap \CC_2$:

$$\deg(D) = 2-2g_2-r - \epsilon(\CC_1\cup \CC_2, \CC_1)$$


\section{On the number of singularities of $\CC$ situated on a line}

In this section, we assume $\CC$ is a reduced curve in $\PP^2$ such that there exists a line $L \in \Irr(\CC)$, and we denote $\CC' \coloneqq \cup_{C \in  \Irr(\CC) \setminus \{L\}}C$.\\

This next exact sequence is  just a particular case of the exact sequence \eqref{eq:sequence_innitial}, where we take $\CC_1 = \CC', \; \CC_2 = L$ a line, and we use the fact that a divisor on a line is determined by its degree.

\begin{equation}
\label{eq:sequence}
0 \rightarrow \mathcal{T}_{\CC'}(-1) \overset{\alpha_L}\longrightarrow  \mathcal{T}_{\CC} \rightarrow \mathcal{O}_L(1-|\CC' \cap L| -\epsilon(\CC, \CC')) \rightarrow 0.
\end{equation}

 By dualization of the exact sequence \eqref{eq:sequence}, 
 we get 
\begin{equation}
\label{eq:dual_line}
0 \rightarrow \mathcal{T}^\vee_{\CC} \rightarrow  \mathcal{T}^\vee_{\CC'}(1) \rightarrow \mathcal{O}_L(|\CC' \cap L| + \epsilon(\CC, \CC') ) \rightarrow 0.
\end{equation}

 Since $\mathcal{T}^\vee_{\CC}=\mathcal{T}_{\CC}(-c_1(\mathcal{T}))=\mathcal{T}_{\CC}(\deg(\CC)-1)$, the exact sequence (\ref{eq:dual_line}) becomes 
\begin{equation}
\label{eq:dual_line2}
0 \rightarrow \mathcal{T}_{\CC} \rightarrow  \mathcal{T}_{\CC'} \rightarrow \mathcal{O}_L(|\CC' \cap L| + \epsilon(\CC,\CC') +1-\deg(\CC)) \rightarrow 0.
\end{equation}

Let $c_1(\mathcal{T}_{\CC}), \; c_2(\mathcal{T}_{\CC})$ be the first, respectively the second Chern number of $\mathcal{T}_{\CC}$ and  let
$$c_{\mathcal{T}_{\CC}}(t) = 1 + c_1(\mathcal{T}_{\CC}) t + c_2(\mathcal{T}_{\CC}) t^2$$
be  the Chern polynomial associated to $\mathcal{T}_{\CC}$.\\

Let us recall a very useful splitting result from \cite{FV}.

\begin{lemma}
\label{lemma:c2_split} (\cite[Lemma 3.3]{FV})
Let $\mathcal{E}$ be a rank $2$ vector bundle on $\PP^2$  such that $c_1(\mathcal{E}) = c \leq 0$ and $c_2(\mathcal{E})  =0$. Then, for any line $L \subset \PP^2$, the restriction  $\mathcal{E}|_L$ of $\mathcal{E}$ to $L$ splits as
 $$\mathcal{E}|_L = \mathcal{O}_L(s) \oplus \mathcal{O}_L(c-s)$$ for some non negative integer $s$ depending on $L$.
\end{lemma} 

\begin{remark}
\label{lemma:c_*_formulas}
Let $\mathcal{E}$ be a rank $2$ vector bundle on $\PP^2$ . It is well known from classical theory on vector bundle Chern classes that  for all integers $k$, 
$$c_1(\mathcal{E}(k)) = 2k +  c_1(\mathcal{E}) \; \; \mathrm{and}\; \; c_2(\mathcal{E}(k)) = k^2 + k c_1(\mathcal{E}) + c_2(\mathcal{E}).$$
\end{remark}

\subsection{Proof of  Theorem \ref{thm:curve_line_comb_chernA}}
 From the hypothesis we have that  $c_1(\mathcal{T}_{\CC}) = -a-b$ and $c_2(\mathcal{T}_{\CC}) = ab$.
By Remark \ref{lemma:c_*_formulas}, $c_1(\mathcal{T}_{\CC}(a)) = a-b$ and $c_2(\mathcal{T}_{\CC}(a)) =0$.
So we can apply Lemma \ref{lemma:c2_split} for the vector bundle 
$\mathcal{T}_{\CC}(a)$.  It follows that there exists an integer $s =s(L)\geq 0$ such that 
$$\mathcal{T}_{\CC}|_{L}=\mathcal{O}_L(-b-s) \oplus \mathcal{O}_L(-a+s).$$

From \eqref{eq:sequence}, in particular, we have a surjective map 
$$ \mathcal{T}_{\CC} \twoheadrightarrow \mathcal{O}_L(1-|\CC' \cap L| -\epsilon(\CC,\CC')).$$
 Tensor this map by $\mathcal{O}_L$ to obtain again a surjection with domain $\mathcal{T}_{\CC} \otimes \mathcal{O}_L = \mathcal{T}_{\CC}|_{L}:$
 $$ \mathcal{T}_{\CC}|_{L}  \twoheadrightarrow \mathcal{O}_L(1-|\CC' \cap L| -\epsilon(\CC,\CC')),$$
 i.e. a surjection 
 \begin{equation}
 \label{eq:surj}
 \mathcal{O}_L(-b-s) \oplus \mathcal{O}_L(-a+s) \twoheadrightarrow \mathcal{O}_L(1-|\CC' \cap L| -\epsilon(\CC,\CC')). 
 \end{equation}
 
 The fact that the map from \eqref{eq:surj} is a surjection already implies that $1-|\CC' \cap L| -\epsilon(\CC,\CC')\notin  ]-b,-a[$, i.e. claim (1)(a) holds.\\
 
If either $|\CC' \cap L|+\epsilon(\CC,\CC') = a+1$ or $|\CC' \cap L|+\epsilon(\CC,\CC') = b+1$, then this implies $s=0$, so  $\mathcal{T}_{\CC}|_{L} = \mathcal{O}_L(-b) \oplus \mathcal{O}_L(-a)$, i.e. the splitting type of $\mathcal{T}_{\CC}$ onto $L$ is $(a,b)$. Then, since $c_2(\mathcal{T}_{\CC}) = ab$, by Theorem \ref{thm:pi_L}, $\CC$ is free, which proves claim (1)(b).\\
 
  If there exists a line $L \in  \Irr(\CC)$  such that $|\CC' \cap L|+\epsilon(\CC,\CC') > b+1$ then in \eqref{eq:surj}, because of surjectivity, we necessarily have $s>0$. Then, again by Theorem \ref{thm:pi_L}, $\CC$ is not free, which proves claim (1)(c).\\
  
Finally, take $L \subset \PP^2$ a line such that $L \notin \Irr(\CC)$.  Denote for convenience $k:=|\CC \cap L| + \epsilon(\CC \cup L, \CC)$.
To prove claim (2), consider the exact sequence \eqref{eq:dual_line}, where we substitute $\CC$ by $\CC \cup L$ and $\CC'$ by $\CC$:
  $$
 0 \rightarrow \mathcal{T}^\vee_{\CC \cup L} \rightarrow  \mathcal{T}^\vee_{\CC}(1) \rightarrow \mathcal{O}_L(k)  \rightarrow 0.
 $$
 In particular, we have a surjection $ \mathcal{T}^\vee_{\CC}(1) \twoheadrightarrow \mathcal{O}_L(k)$. 
 Since  $ \mathcal{T}^\vee_{\CC}(1) =\mathcal{T}_{\CC}(1 - c_1(\mathcal{T}_{\CC})) = \mathcal{T}_{\CC}(1+a+b),$ we have in fact a surjection 
\begin{equation}
\label{eq:srj}
 \mathcal{T}_{\CC}(1+a+b)  \twoheadrightarrow \mathcal{O}_L(k).
 \end{equation} 
 
  Recall that, by Lemma \ref{lemma:c2_split} (applied for the vector bundle 
$\mathcal{T}_{\CC}(a)$), there exists an integer $s =s(L)\geq 0$ such that 
$\mathcal{T}_{\CC}|_{L}=\mathcal{O}_L(-b-s) \oplus \mathcal{O}_L(-a+s)$, hence 
$$
\mathcal{T}_{\CC}(1+a+b)|_{L}=\mathcal{O}_L(a+1-s) \oplus \mathcal{O}_L(b+1+s).
$$

Tensoring the map \eqref{eq:srj} by $\mathcal{O}_L$, we get a surjection 
$$
\mathcal{O}_L(a+1-s) \oplus \mathcal{O}_L(b+1+s) \twoheadrightarrow \mathcal{O}_L(k).
$$
 It follows that either $k \leq a+1$ or $k \geq b+1$, i.e. claim (2)(a) holds. 
 If $k=a+1$ or $k=b+1$ then $s=0$ proving claim (2)(b) and if $k<a+1$ then necessarily $s>0$ proving claim (2)(c).
  
$\Box$\\

Under quasihomogeneity assumptions (which is in particular the case for line arrangements), the statement of Theorem \ref{thm:curve_line_comb_chernA}  looks much simpler, mirroring perfectly and extending  the similar statements for arrangements of projective lines, \cite[Corollary 1.2]{A2}  and \cite[Proposition 5.2]{FV}:

\begin{corollary}
\label{cor:quasihom}
Let $\CC$ be a reduced curve in $\PP^2$ such that 
$c_{\mathcal{T}_{\CC}}(t) = (at-1)(bt-1),\; a,b \in \Z_{>0},$
$ a \leq b$. Let $L$ in $\PP^2$ be a line such that $L \in  \Irr(\CC)$ and let $\CC' \coloneqq \bigcup_{C \in \Irr(\CC) \setminus\{ L\} }C$.
Assume moreover  that all singularities of $\CC, \CC'$ situated on $L$ are quasihomogeneous. Then:
\begin{enumerate}
\item   $|\CC' \cap L|\notin (a+1 , b+1)$, where $(\cdot \; , \cdot)$ denotes an open interval. 
\item If  $|\CC' \cap L| = a+1$ or $|\CC' \cap L|= b+1$, then $\CC$ is free with exponents $(a,b)$.
\item  If  $|\CC' \cap L| > b+1 $, then $\CC$ is not free.
\end{enumerate}
\end{corollary}

\begin{proof}
Straightforward from Theorem \ref{thm:curve_line_comb_chernA}, since, in the quasihomogenity assumption, $\epsilon(\CC,\CC') =0$.
\end{proof}

\begin{corollary}
\label{cor:quasihomU}
Let $\CC$ be a reduced curve in $\PP^2$ such that 
$c_{\mathcal{T}_{\CC}}(t) = (at-1)(bt-1),\; a,b \in \Z_{>0},$
$ a \leq b$. Let $L$ in $\PP^2$ be a line such that $L \notin  \Irr(\CC)$.
Assume moreover  that all singularities of $\CC$  and of $\CC \cup L$ situated on $L$ are quasihomogeneous. Then:
\begin{enumerate}
\item   $|\CC \cap L|\notin (a+1 , b+1)$, where $(\cdot \; , \cdot)$ denotes an open interval. 
\item If  $|\CC \cap L| = a+1$ or $|\CC \cap L|= b+1$, then $\CC$ is free with exponents $(a,b)$.
\item  If  $|\CC \cap L| < a+1 $, then $\CC$ is not free.
\end{enumerate}
\end{corollary}

\begin{proof}
Straightforward from Theorem \ref{thm:curve_line_comb_chernA}, since, in the quasihomogenity assumption, $\epsilon(\CC \cup L, \CC) =0$.
\end{proof}

\begin{remark}
\label{rem:gen_line}
Recall that the Chern polynomial of a free curve with exponents $(a, b)$ factors as $(at-1)(bt-1)$. But such a  factorization of the Chern polynomial does not necessarily imply freeness. Consider for instance the projective lines arrangement from \cite[Example 4.139]{OT}.
\end{remark}

We give next an example of an arrangement of conics and lines that exhibits the same property.

\begin{example}
\label{ex:chern_factor}
Consider a curve $\CC$ defined as the union of a smooth conic $C_0$ and $7$ lines defined as follows. Take a pencil of $5$ lines with base point $P$ and $C_0$ a smooth conic that intersects this pencil transversely, so there are $10$ intersection points between the pencil and the conic. Denote by  $\mathcal{S}$ the set of the intersection points between the conic and the pencil.
Take another line that passes through $P$ and it is tangent to $C_0$ at a point $Q \in C_0$. Lastly, take a line that passes through $Q$ such that it intersects the conic $C_0$ into a point $R \in \mathcal{S}$.

 Let $\A$ be the line arrangement defined as the union of all $7$ lines described above. Notice that $|\A \cap C_0| = 11.$
 Moreover, $\A$ is a free arrangement, since it is supersolvable (\cite[Theorem 4.2]{JT}), with exponents $(1, 5)$, see for instance \cite[Example 4.11]{DIM}. 
Notice that $Q$ is quasi-homogeneous as a singular point of $\CC$  (it is a $D_6$ type singularity).
Then, by \cite[Theorem 3.7]{M},  $\CC = \A \cup C_0$ is plus-one generated with exponents $(3, 6)$ and level $6$. It follows that the  Chern numbers of  $\mathcal{T}_{\CC}$ are $c_1(\mathcal{T}_{\CC}) = -8$ and $c_2(\mathcal{T}_{\CC}) = 16$, see for instance \cite[Proposition 2.6]{MP}, so its Chern polynomial 
polynomial factors as
$$
c_{\mathcal{T}_{\CC}}(t) = (4t-1)^2.
$$
We have that  $\{|\CC' \cap L|\; | \; L \in \Irr(\CC), \; L \;line\} =  \{2,3,4,6\}$ - where $\CC' \coloneqq  \bigcup_{C \in \Irr(\CC) \setminus\{ L\} }C$, which means that for any line $L \in \Irr(\CC)$ we are in one of the cases (1)(a) or (1)(c) of Theorem \ref{thm:curve_line_comb_chernA}. In particular, Theorem \ref{thm:curve_line_comb}(i) does not hold in a hypothesis where we substitute freeness by the factorization of the Chern polynomial.

Take a line $L \notin \Irr(\CC)$ such that $P \in L$ and $L$ is transverse to $C_0$. In particular, it follows that $R \notin L$. Since  $|\CC \cap L| = 4$, this shows that Theorem \ref{thm:curve_line_comb}(ii) does not hold in a hypothesis where we substitute freeness by the factorization of the Chern polynomial.
 \end{example}

\section{On the number of singularities of $\CC$ situated on a smooth conic}
\label{sect:conic}
Let $\CC'$ be a reduced curve in $\PP^2$ such that $C_0$ is a smooth conic which is not an irreducible component of $\CC'$,  $\CC \coloneqq \CC' \cup C_0$ and 
\begin{equation}
\label{eq:k}
k \coloneqq  |\CC' \cap C_0|+ \epsilon(\CC, \CC').
\end{equation}

Let $ j : \PP^1 \stackrel{\sim}\rightarrow C_0$ be an isomorphism. Substituting in \eqref{eq:sequence_innitial} $\CC_1 = \CC'$ and $\CC_2 = C_0$ and  tensoring with $\mathcal{O}_{\PP^2}(-1)$, we get the exact sequence:

\begin{equation}
\label{eq:seq_conic}
0 \rightarrow \cT_{\CC'}(-2)  \overset{\cdot f_{C_0}}\rightarrow \cT_{\CC}  \rightarrow (i \circ j)_* \OO_{\PP^1}(-k)\rightarrow 0,
\end{equation}
which translates into: 
\begin{equation}
\label{eq:seq_gen'}
0 \rightarrow \cT_{\CC'}(-2)  \overset{\cdot f_{C_0}}\rightarrow \cT_{\CC}  \rightarrow i_* \OO_{C_0}(-\frac{k}{2})\rightarrow 0,
\end{equation}
where $\OO_{C_0}(-\frac{k}{2})$ denotes the unique line bundle on $C_0$ with degree $-k$.

Dualizing the exact sequence \eqref{eq:seq_gen'}, we get
$$0 \rightarrow \cT_{\CC}^\vee \rightarrow \cT_{\CC'}(-2)^\vee  \rightarrow 
\mathcal{E}xt^1_{\OO_{\PP^2}}(i_* \OO_{C_0}(-\frac{k}{2}), \OO_{\PP^2})\rightarrow 0,\;  \text{i.e.}$$
$$0 \rightarrow \cT_{\CC}(\deg(\CC)-1)   \rightarrow \cT_{\CC'}(-2)(3+\deg(\CC')) \rightarrow i_* \OO_{C_0}(\frac{k}{2}+2)\rightarrow 0,$$
by Remark \ref{lemma:c_*_formulas}. Consequently, keeping in mind that $\deg(\CC) = \deg(\CC')+2$, there is an exact sequence
\begin{equation}
\label{eq:conic_dual}
0 \rightarrow \cT_{\CC}   \rightarrow \cT_{\CC'} \rightarrow i_* \OO_{C_0}(\frac{k}{2}+1-\deg(\CC'))\rightarrow 0. 
\end{equation}

\smallskip

The exact sequences \eqref{eq:seq_gen'} and \eqref{eq:conic_dual} are central to the proofs of Theorems \ref{thm:curve_conic_comb_chern} and \ref{thm:curve_conic_comb_chern_add}. Moreover, we will essentially use Theorem \ref{thm:conic_restr}, which is the analogue of Lemma \ref{lemma:c2_split} for the restriction of a rank $2$ bundle to a smooth conic, and the splitting criterion relative to the restriction to a smooth conic, Theorem \ref{thm:conic_criterion}. So let us present the proofs of these two general results first.

\subsection{ Proof of Theorem \ref{thm:conic_restr}}

According to Schwarzenberger (see lemma 1.2.7 in \cite{OSS}) the bundle $\mathcal{E}$ is semi-stable or unstable because $c_1^2-4c_2=c^2\ge 0.$ This implies that the minimal integer $t$ such that $H^0(\mathcal{E}(t))\neq 0$ verifies $c+2t\le 0.$
This gives an exact sequence:
$$ 0\rightarrow \OO_{\PP^2}\longrightarrow \mathcal{E}(t)\longrightarrow  I_W (2t +c)\rightarrow 0  $$
where  $W$ is a finite subscheme
of length:
$$c_2 (\mathcal{E}(t)) = t(t +c) \ge  0.$$
If $t=0$ then $W$ is empty and this implies $\mathcal{E}=\OO_{\PP^2} \oplus \OO_{\PP^2}(c)$.\\
Let $C$ be a smooth conic ; Tensoring the above exact sequence by $\OO_C$ we get
$$ 0\rightarrow \OO_{C}\longrightarrow \mathcal{E}(t)\otimes \OO_C \longrightarrow I_W (2t +c)\otimes \OO_C\rightarrow 0.  $$
Then two different cases occur:
\begin{enumerate}
    \item if  $C\cap W=\emptyset$ then $I_W (2t +c)\otimes \OO_C=\OO_C(2t+c)$ and since $2t+c\le 0$ this implies $$\mathcal{E}(t)\otimes \OO_C=\OO_C\oplus \OO_C(2t+c).$$
    \item if $C\cap W$ is a divisor on $C$ of length $s$ then $I_W (2t +c)\otimes \OO_C=\OO_C(2t+c-\frac{s}{2})\oplus Q$ where $Q$ is a torsion sheaf on $C$ supported by the intersection scheme (i.e. $H^0(Q)=s$). This induces a surjective map 
$$ \mathcal{E}(t)\otimes \OO_C \longrightarrow \OO_C(2t+c-\frac{s}{2}) \rightarrow 0.  $$
Since $\mathcal{E}(t)\otimes \OO_C =\OO_C(a)\oplus \OO_C(b)$ with $a+b=c+2t$ and $c+2t-\frac{s}{2}<0$ this implies 
$$ \mathcal{E}(t)\otimes \OO_C =\OO_C(\frac{s}{2})\oplus \OO_C(c+2t-\frac{s}{2}).$$
\end{enumerate}

Let us verify now that $t$ is necessarily a negative number. Indeed, assume that $t>0.$ Then $H^0(\mathcal{E})=0$.
By Riemann-Roch, the Euler characteristic 
$$\chi(\mathcal{E})=  \frac{1}{2} (c_1 (\mathcal{E})^2 - 2c_2 (\mathcal{E}) + 3c_1 (\mathcal{E}) + 4)
= \frac{1}{2} (c^2 + 3c + 4)$$ is positive, hence $H^2 (\mathcal{E}) \neq 0$, so $H^0 ( \mathcal{E}(-c-3)) \neq 0$ by Serre duality, indeed 
$ \mathcal{E}^{\vee} =\mathcal{E}(-c)$.
Therefore $t > 0$ implies $t \le -c- 3$. But since $c_2 (\mathcal{E}(t)) = t(t +c) \ge  0$, the hypothesis $t > 0$ implies $t\ge -c$, a contradiction.

\smallskip

The theorem is proved for $\frac{r}{2}=-t+\frac{s}{2}.$ 

\smallskip

Let us recall that if $t=0$ then $W$ is empty (so $s=0$) and this implies $$\mathcal{E}=\OO_{\PP^2} \oplus \OO_{\PP^2}(c).$$ Then $t<0$ which implies that $r\ge 2$.
$\Box$

\subsection{Proof of Theorem \ref{thm:conic_criterion}}
The 'if' part of the statement is immediate. In fact, if $\cE$ splits as $\cE =\OO_{\PP^2}(-a) \oplus  \OO_{\PP^2}(-b)$, then $c_{\mathcal{E}}(t) = (at-1)(bt-1)$ and, moreover, 
for any smooth conic  $C_0 \subset \PP^2$, we get  $\cE|_{C_0} =\OO_{C_0}(-a) \oplus  \OO_{C_0}(-b)$.

Let us prove now the 'only if part'. Assume without loss of generality that $a \leq b$. From Remark \ref{lemma:c_*_formulas}, we know that $c_1(\cE(a)) = a-b$ and $c_2(\cE(a)) = 0$, which implies that $\cE(a)$ is 
unstable or semistable.
Then, by the proof of the previous Theorem \ref{thm:conic_restr}, we have $H^0(\cE(a))\neq 0$. We can't have $H^0(\cE(a-1))\neq 0$ because then we would have 
$$ 0\rightarrow \OO_{\PP^2}\longrightarrow \mathcal{E}(a-1)\longrightarrow  I_Z (-b-2+a)\rightarrow 0,$$ where  $Z$ is a finite subscheme of length $c_2(\mathcal{E}(a-1)) = b+1-a >  0$,
and restricting to $C_0$ we would have
$$\OO_{C_0} \rightarrow \OO_{C_0}(-1) \oplus \OO_{C_0}(-b+a-1),$$
which is impossible because $-b-1+a<0$.
Then $a$ is minimal with the property $H^0(\cE(a))\neq 0$ and we have 
$$0 \rightarrow \OO_{\PP^2} \rightarrow \cE(a)\rightarrow I_Z(a-b) \rightarrow 0.$$ But then $c_2(\cE(a))=\text{length}(Z)=0$ implies that $Z$ is empty, so we obtain 
$$0 \rightarrow \OO_{\PP^2} \rightarrow \cE(a)\rightarrow \OO_{\PP^2}(a-b) \rightarrow 0,$$
hence $\cE(a)= \OO_{\PP^2} \oplus \OO_{\PP^2}(a-b)$, i.e. $\cE= \OO_{\PP^2}(-a) \oplus \OO_{\PP^2}(-b)$.
$\Box$\\

\subsection{Proof of Theorem \ref{thm:curve_conic_comb_chern}}
Since  $c_1(\mathcal{T}_{\CC}) = -a-b$ and $c_2(\mathcal{T}_{\CC}) = ab$, then, by Remark \ref{lemma:c_*_formulas}, $c_1(\mathcal{T}_{\CC}(a)) = a-b$ and $c_2(\mathcal{T}_{\CC}(a)) =0$.
Hence we can apply Theorem \ref{thm:conic_restr} for the vector bundle 
$\mathcal{T}_{\CC}(a)$. Then there exists an integer $r \geq 0$ such that 
$$\mathcal{T}_{\CC}|_{C_0}=\mathcal{O}_{C_0}(-b-\frac{r}{2}) \oplus \mathcal{O}_{C_0}(-a+\frac{r}{2}).$$
From \eqref{eq:seq_conic} we have a surjective map 
$$ \mathcal{T}_{\CC} \rightarrow i_*\mathcal{O}_{C_0}(-\frac{k}{2})\rightarrow 0,$$
which, after tensoring by $\mathcal{O}_{C_0}$, produces again a surjection
$$\mathcal{T}_{\CC} \otimes \mathcal{O}_{C_0} =  \mathcal{T}_{\CC}|_{C_0}  \twoheadrightarrow \mathcal{O}_{C_0}(-\frac{k}{2}),$$
 i.e. a surjection 
\begin{equation}
\label{eq:surj_conic}
\mathcal{O}_{C_0}(-b-\frac{r}{2}) \oplus \mathcal{O}_{C_0}(-a+\frac{r}{2}) \twoheadrightarrow \mathcal{O}_{C_0}(-\frac{k}{2}). 
\end{equation} 
 This surjection implies that  
 \begin{equation}
 \label{eq:k_set}
 k=2b+r \; \text{ or } \; k\leq 2a-r
 \end{equation}
 hence $k \geq 2b \text{ or } k \leq 2a
 $, which implies the points (1) and (2) of the  theorem.

Let us now prove the (a) 'moreover' parts of the statement.
When $k=2a$ or $k=2b$ (resp. $k=2a-1$ or $k=2b-1$) the surjection  in
(\ref{eq:surj_conic}) implies $r=0$ (resp. $0\le r\le 1$ but $r=1$ being impossible this implies also $r=0$). By Theorem \ref{thm:conic_criterion} this proves that $\CC$ is free. 

To prove the (b) 'moreover' parts of the theorem, consider $k\le 2a-2$ (so $k<2a$ in the even case or $k<2a-1$ in the odd case). Then two possibilities occur : if $r=0$ the curve is free by Theorem \ref{thm:conic_criterion}, if $r>1$  the splitting of $\mathcal{T}_{\CC}$ is not imposed by the surjective map:
$$\mathcal{T}_{\CC}|_{C_0}\twoheadrightarrow \mathcal{O}_{C_0}(-\frac{k}{2}).$$
If $$\mathcal{T}_{\CC}|_{C_0}=\mathcal{O}_{C_0}(-b) \oplus \mathcal{O}_{C_0}(-a)$$
the curve is free; if 
$$\mathcal{T}_{\CC}|_{C_0}=\mathcal{O}_{C_0}(-b-\frac{r}{2}) \oplus 
\mathcal{O}_{C_0}(-a+\frac{r}{2})$$
with $r\ge 2$ the curve is not free.

Finally, let us prove the (c) 'moreover' parts of the theorem.
If we assume to the contrary that $\CC$ is free with exponents $(e_1, e_2), \; e_1 \leq  e_2$, then we have a factorization  $c_{\mathcal{T}_{\CC}}(t) = (e_1t-1)(e_2t-1)$ of the Chern polynomial of $\mathcal{T}_{\CC}$. Since by hypothesis $c_{\mathcal{T}_{\CC}}(t) = (at-1)(bt-1)$, this implies $e_1=a$ and $e_2=b$. 
Consider the case $k=2m$. 
Then, by Theorem \ref{thm:curve_conic_comb_intro}, $m \leq e_2 = b$, contradiction.
In the case $k=2m+1$, Theorem \ref{thm:curve_conic_comb_intro} implies $m \leq e_2-1$, i.e.  $m \leq b-1$, again a contradiction. 
In conclusion, in both cases, $\CC$ cannot be free.
$\Box$

\subsection{Proof of Theorem \ref{thm:curve_conic_comb_chern_add}}

The proof goes exactly as the one of Theorem \ref{thm:curve_conic_comb_chern}. Consider the exact sequence \eqref{eq:conic_dual}, where we substitute $\CC$ by $\CC \cup C_0$ and $\CC'$ by $\CC$:
$$0 \rightarrow \cT_{\CC \cup C_0}   \rightarrow \cT_{\CC} \rightarrow i_* \OO_{C_0}(\frac{k}{2}+1-\deg(\CC))\rightarrow 0$$
After tensoring the surjective map from the above exact sequence by $\mathcal{O}_{C_0}$, we get again a surjection
$$\mathcal{T}_{\CC} \otimes \mathcal{O}_{C_0} =  \mathcal{T}_{\CC}|_{C_0}  \twoheadrightarrow \mathcal{O}_{C_0}(\frac{k}{2}+1-\deg(\CC)),$$
 i.e., by Theorem \ref{thm:conic_restr}, a surjection 
\begin{equation}
\label{eq:surj_add}
\mathcal{O}_{C_0}(-b-\frac{r}{2}) \oplus \mathcal{O}_{C_0}(-a+\frac{r}{2}) \twoheadrightarrow \mathcal{O}_{C_0}(\frac{k}{2}+1-\deg(\CC)) 
\end{equation}
for some positive integer $r$ depending on $C_0$.
 Since $\deg(\CC) = a+b+1$, this surjection tensor by $\mathcal{O}(a+b)$ becomes
 \begin{equation}
\label{eq:surj_add2}
\mathcal{O}_{C_0}(a-\frac{r}{2}) \oplus \mathcal{O}_{C_0}(b+\frac{r}{2}) \twoheadrightarrow \mathcal{O}_{C_0}(\frac{k}{2}) 
\end{equation}

 and this  implies that  
$k=2a-r \; \text{ or } \; k\geq 2b+r$,
 hence $k \geq 2b \text{ or } k \leq 2a
 $, which proves assertions (1) and (2) of the theorem.

 \smallskip

--- Assume that $k=2m$. If $m=a$ or $m=b$ necessarily $r=0$ which implies freeness. \\ If $m<a$, say $m=a-s$ with $s>0$ then the surjection imposes 
 $$\mathcal{T}_{\CC}|_{C_0}(a+b)=\mathcal{O}_{C_0}(a-s) \oplus \mathcal{O}_{C_0}(b+s), $$ hence $\CC$ is not free.\\
 If $m>b$ say $m=b+s$ then the surjection imposes only $\frac{r}{2}\le s$ ; if $r=0$, $\CC$ is free, if $r\neq 0$ it is not.\\
 --- Assume now that $k=2m+1$. If $m=a$ the surjection onto $\mathcal{O}_{C_0}(a+\frac{1}{2})$ imposes $b=a$ and $r=0$ proving that $\CC$ is free. If $m=b$ then there is a surjection onto
$ \mathcal{O}_{C_0}(b+\frac{1}{2})$. This imposes $r\le 1$ but since $r=1$ is impossible one gets $r=0$ that is $\CC$ is free.\\
If $m<a$ say $m=a-s$ with $s>0$ then the surjection onto $\mathcal{O}_{C_0}(a-s+\frac{1}{2})$ imposes $r=2s-1>0$ i.e. $\CC$ is not free. \\
If $m>b$ say $m=b+s$ with $s>0$ then the surjection onto $\mathcal{O}_{C_0}(b+s+\frac{1}{2})$ does not impose any condition on $r$; if $r=0$ the curve is free, if $r\neq 0$ it is not.
 $\Box$

 \smallskip
 
The next two examples illustrate  the situation in Theorem \ref{thm:curve_conic_comb_chern} when 
$m<a$ and $\CC$ is not free (Example \ref{example-notfree}) and when 
$m<a$ and $\CC$ is free (Example \ref{example-free}).

\begin{example}
\label{example-notfree}
    Let us consider the curve $\CC$ formed by the union of two smooth conics $C_1$ and $C_2$ meeting in four points $a,b,c,d$, the reducible conic $(ac)\cup (bd)$ belonging to the pencil generated by $C_1$ and $C_2$ and the line $(ab)$. Since $(ab)$ is an irreducible component of the reducible conic $(ab)\cup (cd)$ which belongs also to the same pencil one has an exact sequence (see Theorem 2.8 in \cite{VaE})
$$0 \rightarrow \OO_{\PP^2} (-2)\rightarrow \cT_{\CC}   \rightarrow \mathcal{I}_Z(-4) \rightarrow 0.$$
The non zero section of $\mathrm{H}^0(\cT_{\CC} (2))$ is the ``canonical section'' (defined in \cite{VaE}) associated to the pencil and $Z$ consists in the single point $(ad)\cap (bc).$ Then $c_1(\cT_{\CC}(2))=-2$ and $c_2(\cT_{\CC} (2))=1$, so $\cT_{\CC} $ has the same Chern classes than $\OO_{\PP^2} (-3)^2$ but $\cT_{\CC} \neq \OO_{\PP^2} (-3)^2.$

Removing one smooth conic $C_1$ we get a new curve $\CC'$ and an exact sequence 
$$0 \rightarrow \OO_{\PP^2} (-2)\rightarrow \cT_{\CC'}   \rightarrow \mathcal{I}_Z(-2) \rightarrow 0$$
where $Z$ is the same than before.
This gives without any difficulty
$$0 \rightarrow \cT_{\CC} \rightarrow \cT_{\CC'}   \rightarrow \OO_{C_1}(-2) \rightarrow 0.$$
With the notations of Theorem \ref{thm:curve_conic_comb_chern} we have $k=4$, so $m=2$, and $a=b=3$; this gives an example where $m<a$ and $\CC$ is not free.
\end{example}

\begin{example}
\label{example-free}
Consider the curve from \cite[Example 4.4]{M},  $$\CC:  (x^2 + 2xy + y^2 + xz)(x^2 + xz + yz)(x^2 + xy + z^2)(x+y-z)y(x+z)
(2x+y)(x^2 - y^2 + xz + 2yz)\cdot $$
$$
(x^2 + 2xy - xz + yz)=0.$$


The curve $\CC$ is free with exponents $(a,b) = (6,7)$. Take the smooth conic $C_0 \in \Irr(\CC)$, 
$$C_0: x^2 - y^2 + xz + 2yz = 0.$$ Denote $\CC'\coloneqq \CC \setminus \{C_0\}$. Then $|\CC' \cap C_0|=6$. 
As singularities of $\CC$, these six points are  ordinary singularities of multiplicity 5, 6, 7 and three ordinary singularities of multiplicity 4.
Only two of these singularities are not quasihomogeneous, $P = [0:0:1]$, an ordinary singularity with $6$ branches, respectively $Q = [1:-2:-1]$, an ordinary singularity with $7$ branches and one computes $\epsilon(\CC, \CC')_{P}=1, \;  \epsilon(\CC, \CC')_{Q} =1$. 
Hence $\epsilon(\CC, \CC')=2.$
Then, in the notations of Theorem \ref{thm:curve_conic_comb_chern}, $k = 2m = 8 $, $m=a-2 < a$.
\end{example}

\subsection{The addition-deletion of a smooth conic  to a curve, revisited}
\label{subsect:add_dell}

\smallskip

 In \cite{M} the result of  addition-deletion of a smooth conic to a free curve is described using the exact sequence \eqref{eq:seq_conic}. We suggest here an alternative approach, via the exact sequence \eqref{eq:seq_gen2} below.
 \\

 Let $\CC$ be a reduced curve in $\PP^2$ and $C_0$ a smooth conic  which is not a component of $\CC$.
We get the exact sequence:

\begin{equation}
\label{eq:seq_gen2}
0 \rightarrow \cT_{\CC\cup C_0}  \rightarrow \cT_{\CC}  \rightarrow \OO_{C_0}(- \frac{k_0}{2} )\rightarrow 0.
\end{equation}
From this exact sequence we can deduce a relation between invariants of the curves $\CC$ and $\CC\cup C_0$ and the integer $k_0$. Indeed, the Chern polynomial $c_{\cT_{\CC}}(t)$ being the product (modulo $t^3$) of the Chern polynomials  $c_{\cT_{\CC\cup C_0}}(t)$ and $c_{\OO_{C_0}(- \frac{k_0}{2})}(t)$
we obtain the following relation:
\begin{lemma} \label{k_0 and k_1} 
 $$k_0=-2+\tau (\CC\cup C_0, \CC)-2 \times \mathrm{deg}(\CC),$$ 
 where $\tau (\CC\cup C_0, \CC):=\sum_{p\in \CC\cap C_0} [\tau_p(\CC\cup C_0)-\tau_p(\CC )]$
\end{lemma}
\begin{proof}
It is a direct computation. Indeed, the Chern polynomial 
of $\OO_{C_0}(- \frac{k_0}{2} )$ is 
$$ 1+2t+(k_0+4)t^2$$ and if $D$ is a reduced divisor the Chern polynomial of $\cT_D$ is
$$1+(1-\deg (D))\,t +( (\mathrm{deg}(D)-1)^2-\tau(D))\,t^2$$ where $\tau(D)$ is the total Tjurina number of $D$.
\end{proof}

Dualizing the exact sequence \eqref{eq:seq_gen2} we find again (after a shift)
\begin{equation}
\label{eq:seq_gen3'}
0 \rightarrow \cT_{\CC}(-2)  \rightarrow \cT_{\CC\cup C_0}  \rightarrow \OO_{C_0}(\frac{-2\deg(\CC)+2+k_0}{2} )\rightarrow 0. 
\end{equation}
We had seen before, see \eqref{eq:seq_gen'}, that 
$$ 0 \rightarrow \cT_{\CC}(-2)  \rightarrow \cT_{\CC\cup C_0}  \rightarrow \OO_{C_0}(-\frac{k}{2} )\rightarrow 0 $$
where $k :=|\CC \cap C_0|+\epsilon(\CC \cup C_0, \CC)$. 
This gives the identity 
\begin{equation}
\label{eq:k_sum}
k+k_0 = 2(\deg(\CC)-1),
\end{equation}
i.e.
$$ k_0+2=2\deg(\CC)-|\CC \cap C_0|-\epsilon(\CC \cup C_0, \CC).$$
According to the previous lemma this gives also
$$\epsilon(\CC \cup C_0, \CC)= - \tau (\CC\cup C_0, \CC)+ 4\deg(\CC)-|\CC \cap C_0|. $$

\begin{prop}
\label{prop:mu_formula}
Let $\CC$ be a curve and $C_0$ a smooth conic such that $C_0 \notin \Irr(\CC)$. Then
\begin{equation}
\label{eq:new_mu_P_id}
    \sum_{P \in \CC \cap C_0}[\mu_P(\CC \cup C_0) - \mu_P(\CC)+1] = 4 \deg(\CC).
    \end{equation}
\end{prop}

\begin{proof}
Just replace in  \eqref{eq:k_sum} the parameters $k, k_0$ by their formulas recalled above.
Indeed 
$$k+k_0=|\mathcal{C}\cap C_0| +\epsilon(\mathcal{C} \cup C_0, \CC)-2+\tau(\mathcal{C} \cup C_0, \mathcal{C})-2\mathrm{deg}(\mathcal{C})$$
$$= |\mathcal{C}\cap C_0| +\sum_{P\in \mathcal{C}\cap C_0} [\epsilon_p(\mathcal{C} \cup C_0)-\epsilon_p(\mathcal{C})]-2+
\sum_{p\in \CC\cap C_0} [\tau_p(\CC\cup C_0)-\tau_p(\CC )]
-2\mathrm{deg}(\mathcal{C})$$
$$= |\mathcal{C}\cap C_0| +\sum_{P\in \mathcal{C}\cap C_0} [\mu_P(\CC \cup C_0) - \mu_P(\CC)]-2-2\mathrm{deg}(\mathcal{C}).$$   
\end{proof}

\begin{corollary}
\label{cor:delta_inv}
Let $\CC$ be a curve, $C_0$ a smooth conic such that $C_0 \notin \Irr(\CC)$ and let $\delta$ be the delta-invariant. Then
$$
  \sum_{P \in \CC \cap C_0}[\delta_P(\CC \cup C_0) - \delta_P(\CC)] = 2 \deg(\CC).
$$
\end{corollary}

\begin{proof}
Recall the Milnor formula $\mu = 2 \delta - r +1$, where $r$ is the number of branches at a plane curve singularity $P$ and $\delta, \mu$ are the delta invariant, respectively the Milnor number, at $P$. Then the conclusion follows from \eqref{eq:new_mu_P_id} by a direct computation.
\end{proof}

\bigskip

\end{document}